\documentclass[a4paper,11pt]{ijamas}

\usepackage[pdftex]{graphicx}

\usepackage{url}

\begin{document}


\title{Numerical Solution of the Tomography Problem in the Presence of Obstacles}

\author{\textbf{Kamen Lozev}}

\date{kamen@ucla.edu}

\maketitle


\begin{abstract}
\noindent \emph{We study numerical methods of tomography in domains with a reflecting obstacle. 
It will be shown that tomography with sets containing both broken rays, i.e. rays reflecting at 
the obstacle, as well as unbroken rays, has a smaller error between the original and 
reconstructed image compared to classical tomography methods.
}

\medskip


\medskip


\end{abstract}


\section{Introduction}

In biology, geophysics, oceanography, communications and other areas of science and technology it is required to find the structure 
of a domain such as the human body, the Earth's interior, the oceans or the atmosphere from measurements taken at the domain's boundary. 
This problem can be approached with the methods of tomography. When the domain contains an obstacle new problems need to be addressed. 
In order to solve these problems, we study broken ray tomography. This work develops numerical methods using broken rays, or rays reflecting 
at the obstacle, for solving the tomography problem. It will be shown that broken ray tomography leads to tomographic 
imaging with reduced error and better numerical stability compared to classical tomography in the presence of reflecting obstacles.

\bigskip

Let $u(x)$ designate acoustic wave propagation from point A to point B. Here $u(x)$ denotes the pressure at point $x \in \mathbb{R}^2$.
This process is governed by the wave equation
$$ u_{tt} - c(x)^2 \Delta{u}=0$$
where $c(x)$ is the speed of sound at point $x \in \mathbb{R}^2$. The travel time of wave propagation is 
$$T(A,B)=\int_{\gamma_1}\frac{ds}{c(x(s))}$$
where the curve $\gamma_1$ is the acoustic geodesic from A to B. 

\bigskip

Assume that $c(x)$ is close to a constant: $c(x)=c_o+\epsilon(x)$ where $\epsilon(x)$ is small compared to $c_o$.
Then we will consider the approximation of $\gamma_1$ as the straight line segment $\tilde{\gamma_1}$. It is shown in \cite{RO}
that this linearization and more generally computing the travel time of u(x) over the known geodesic of the known
speed $c_o$ can be justified under certain conditions. The proof given in \cite{RO} is based on considering a perturbation parameter $\lambda$ and 
representing $c(x)=c_o+\lambda \epsilon(x)$.

\bigskip
 
Let $f(x)$ be a continuous function in $\Omega_0$, where $\Omega_0$ is a compact simply-connected set in $\mathbb{R}^2$ with a smooth boundary. Consider $\partial{\Omega_0}$ as the observation 
boundary of $\Omega_0$. Suppose we are given all integrals $\int_{\gamma} f(l) dl = C_{\gamma}$ where $\gamma$ are all straight line segments or unbroken 
rays in $\Omega_0$ that have both of their endpoints in $\partial{\Omega_0}$. The classical Tomography Problem is to find $f(x)$ in $\Omega_0$ knowing 
the values of all integrals $C_{\gamma}$. In the case of a domain $\Omega_0$ without an obstacle $\Omega_1 \subset \Omega_0$, this problem is widely studied 
theoretically and numerically \cite{Na1,KS,NW,F}. When there are obstacles present, this problem is much less studied. Some theoretical work is done 
in \cite{E1,E2,E3,E4} for domains with one obstacle and with both broken rays, i.e. rays reflecting at the obstacle, and unbroken rays. 
A key result from \cite{E1} is that the tomography problem in the presence of one reflecting obstacle is well-posed. 
In this work, we implement a numerical method for solution of the tomography problem in a domain with one convex obstacle with piece-wise smooth boundary 
using unbroken and broken rays. If an obstacle $\Omega_1 \subset \Omega_0$ is present then we have the problem of recovering $f(x)$  in $\Omega_0 \backslash \Omega_1$ with broken 
and unbroken rays $\gamma$. This problem is called the Broken Ray Tomography Problem\cite{E1}. 

\bigskip

In a basic tomography setup transmitters and receivers of wave signals are placed at the domain's boundary $\partial{\Omega_0}$. 
Signals are generated by the transmitters and received by the receivers. Travel times for signal propagation from transmitters to 
receivers are measured and, as discussed, these travel times $T(A, B)$ are the values of line integrals of a function $f(x)=\frac{1}{c(x)}$ 
where $c(x)>0$ is the speed of sound at point $x \in \Omega_0 \backslash \Omega_1$. This measurement procedure gives the $C_{\gamma}$ data 
for solving the Tomography Problem by relating signal travel times to the values of line integrals of $f(x)$. Given sufficient data, $f(x)$ 
and from here the velocity $c(x)$ are computed with tomographic reconstruction algorithms \cite{Na1,KS,NW,F}. 

\bigskip

In nature and in technology, signals are often acoustic or electromagnetic waves. The propagation speed of electromagnetic waves 
is very high and this leads to very small travel times. In order to work with relatively larger travel times, I will start the 
description of the physical problem by choosing high frequency acoustic waves or ultrasound as the signal carrier for obtaining 
tomography travel time data. The choice of ultrasound, or acoustic waves with frequency above 20 kHz is crucial. The analysis 
and methods developed for ultrasound will be extended to other carriers and especially laser beams. The physical model for this 
work is linear ultrasonic wave propagation with reflection; a method for non-linear reconstructive ultrasound tomography is developed 
in \cite{Sch}.
 
\bigskip

The acoustic geodesic for constant speed of sound is a straight line when there is no obstacle. Therefore, we consider two cases. 
In the first case, $\gamma = \tilde{\gamma_1}$ is an unbroken ray. 
In the case of a broken ray, $\gamma=\tilde{\gamma_1} \bigcup \tilde{\gamma_2}$ 
is the union of two straight line segments that intersect at a reflection point at the obstacle. 
Reflection is mirror-like i.e. the angle of incidence is equal to the angle of reflection.

\bigskip

For unbroken rays 
$$ T(A,B)=\int_{\gamma}\frac{ds}{c_o + \epsilon{(x(s))}} = \int_{\tilde{\gamma_1}}\frac{ds}{c_o + \epsilon{(x(s))}} $$ and this 
model leads to the classical tomography problem with $f(x)=\frac{1}{c_o+\epsilon(x)}$.

\bigskip

For broken rays and a known obstacle, we know that the acoustic wave $u(x)$ propagates along the known straight line segments 
$\tilde{\gamma_1}$ and $\tilde{\gamma_2}$. 

\bigskip
 
Then $$ T(A,B)=\int_{\gamma}\frac{ds}{c_o + \epsilon{(x(s))}}=\int_{\tilde{\gamma_1}}\frac{ds}{c_o + \epsilon{(x(s))}} + \int_{\tilde{\gamma_2}}\frac{ds}{c_o+\epsilon{(x(s))}} $$

\bigskip

This model leads to the Broken Ray Tomography Problem. Reflection and the presence of an obstacle are unique to the broken ray model 
compared to the classical tomography problem. The broken ray model is based on travel time computation and line integrals over broken and 
unbroken rays; the line integrals of classical tomography are over straight line segments. Moreover, the broken ray model requires tomographic 
reconstruction in the presence of an obstacle. For both models, we choose $f(x)=\frac{1}{c_o+\epsilon(x)}$. This physical model corresponds to 
a wave propagation environment that has almost constant wave propagation speed. Examples of such environments could be pollution plumes in the 
oceans or the atmosphere or regions of high concentration of minerals in water. Tomographic imaging of environmental pollution such as the recent 
oil spill in the Gulf of Mexico is a target application of broken ray tomography. The next chapter gives a geometric optics solution of the wave 
equation concentrated in a small neighborhood of a broken ray, which shows that we can work directly with broken rays when modeling mirror-like 
reflection of waves propagating approximately along broken rays.

\section{Geometric Optics Construction of a Broken Ray}
 
Following the approach in \cite{E1,E2}, I will construct geometric optics solutions of the wave equation, see \cite{E5}, in a small neighborhood of a broken ray. 

Consider a wave, or signal, described by the wave equation
\begin{equation}\label{waveequation}
u_{tt} - c^2\Delta{u}=0
\end{equation}

where $c=c_o>0$ is the constant speed of wave propagation and

$$ u |_{\partial{\Omega_1}} = 0 $$ 

for $t>0$ and $\Omega_1(t) \subset \Omega_0$, where $\Omega_1(t)$ is a moving convex obstacle with a piece-wise smooth boundary. We discussed in the introduction that under certain conditions such as small variations of the speed of wave propagation, we can consider wave propagation along the geodesic of a wave with constant speed $c_o$ and that is why we consider a form of the wave equation with constant c.

\bigskip

The geometric optics approach to finding a solution of the wave equation that approximates a broken ray is based on the idea of looking for a solution of the form   
\begin{equation}\label{gosolution}
u(x,t)=e^{-ikct+ikw \cdot x}\sum_{p=0}^N \frac{a_p}{k^p} + e^{-ikct+i k \psi(w, x) }\sum_{p=0}^N \frac{b_p}{k^p} + U_N(x,t,w)
\end{equation}

\bigskip

Starting from a plane wave solution $e^{-ikct+ikw \cdot x}$ of the wave equation, 
where k is large, i.e. $k \rightarrow \infty$, and $w=(w_1,w_2)$, $|w|=1$, is the unit direction vector of 
wave propagation, construct a geometric optics solution of \ref{waveequation} where

\bigskip

$$ e^{-ikct+ikw \cdot x}\sum_{p=0}^N \frac{a_p(x,t,w)}{k^p} $$
approximates the broken ray before reflection as a sum of plane waves with varying amplitudes
and 
$$e^{-ikct+ik \psi(w, x) }\sum_{p=0}^N \frac{b_p(x,t,w)}{k^p}$$
approximates the broken ray after reflection as a sum of waves with a nonlinear phase $\psi(w,x)$.
The term $U_N(x,t,w)$ makes the solution $u(x,t)$ exact. 

\bigskip
 
Plugging the target solution in the wave equation leads to the equation
$$ (\frac{\partial^2}{\partial{t^2}}- c^2\Delta)(e^{-ikct+ikw \cdot x}\sum_{p=0}^N \frac{a_p}{k^p} + e^{-ikct+i k \psi(w, x)}\sum_{p=0}^N \frac{b_p}{k^p}+ U_{N}(x,t,w)) = 0 $$

The amplitudes $a_p$ and $b_p$ are constructed for $p=0,1,2,...$ by plugging in the wave equation the target solution for $N=0,1,2...$. In the resulting expansion, equate to 0 the sum of terms for equal powers of k for each
of the two wave coefficients $e^{-ikct+ikw \cdot x}$ and $e^{-ikct+ik \psi(w, x) }$. This leads to a set of equations for the amplitudes $a_p(x,t,w)$ and $b_p(x,t,w)$. For example, for ${a_0}(x,t,w)$ and $N=0$ this step leads to the equation
\begin{equation}
\begin{split}
(\frac{\partial^2}{\partial{t^2}}- c^2\Delta)(e^{-ikct+ikw \cdot x}a_{0}(x,t,w)) = \\
-k^2c^2e^{-ikct+ikw \cdot x}{a_0}(x,t,w) + e^{-ikct+ikw \cdot x}\frac{\partial^2{{a_0}(x,t,w)}}{\partial{t^2}} - \\ 
-2ikce^{-ikct+ikw \cdot x}\frac{\partial{{a_0}(x,t,w)}}{\partial{t}} + \\
- c^2( -k^2 {w_1}^2e^{-ikct+ikw \cdot x}{a_0}(x,t,w) + e^{-ikct+ikw \cdot x}\frac{\partial^2{{a_0}(x,t,w)}}{\partial{{x_1}^2}} + \\
+2ik{w_1}e^{-ikct+ikw \cdot x}\frac{\partial{{a_0}(x,t,w)}}{\partial{{x_1}}} + \\
-k^2 {w_2}^2e^{-ikct+ikw \cdot x}{a_0}(x,t,w) + e^{-ikct+ikw \cdot x}\frac{\partial^2{{a_0}(x,t,w)}}{\partial{{x_2}^2}} -\\ 
+2ik{w_2}e^{-ikct+ikw \cdot x}\frac{\partial{{a_0}(x,t,w)}}{\partial{{x_2}}} ) = 0
\end{split}
\end{equation}

\bigskip
  
where $w=(w_1,w_2)$, $|w|=1$, is the unit direction vector of wave propagation. Cancelling the term $e^{-ikct+ikw \cdot x}$ gives  

\bigskip

\begin{equation}
\begin{split}
-k^2c^2{a_0}(x,t,w) + \frac{\partial^2{{a_0}(x,t,w)}}{\partial{t^2}} - 
2ikc\frac{\partial{{a_0}(x,t,w)}}{\partial{t}} + \\
- c^2( -k^2 {w_1}^2{a_0}(x,t,w) + \frac{\partial^2{{a_0}(x,t,w)}}{\partial{{x_1}^2}}
+2ik{w_1}\frac{\partial{{a_0}(x,t,w)}}{\partial{{x_1}}} - \\
-k^2 {w_2}^2{a_0}(x,t,w) + \frac{\partial^2{{a_0}(x,t,w)}}{\partial{{x_2}^2}} + 
2ik{w_2}\frac{\partial{{a_0}(x,t,w)}}{\partial{{x_2}}} ) = 0
\end{split}
\end{equation}

\bigskip

Collecting terms and using the fact that $w^2={w_1}^2+{w_2}^2$ and that the coefficient for k in the above polynomial of k must be 0 leads to
$$ \frac{\partial{a_0}(x,t,w)}{{\partial{t}}} + cw \cdot \frac{\partial{a_0(x,t,w)}}{\partial{x}}=0$$
 
Make the change of variables $x = sw + \tau w_\bot$, where $w_\bot=(-w_2,w_1)$, and $w \cdot w_\bot =0$. Then
$$s=w \cdot x$$
$$\tau=w_\bot \cdot x$$
Then the last equation can be written in the form
$$\frac{\partial{a_0(x,t,w)}}{\partial{t}}+c\frac{\partial{a_0(x,t,w)}}{\partial{s}}=0$$   

The solution of this equation by the method of characteristics is
$$ a_0(s, \tau, t, w) = f(s-ct,\tau, w) $$

where $f$ is an arbitrary smooth function. 

\bigskip

In order to localize the solution in a neighborhood of the broken ray, the solution is multiplied by cut-off functions.

\bigskip

Continuing in this manner and plugging into the wave equation the sum 

$e^{-ikct+ikw \cdot x}\sum_{p=0}^N \frac{a_p(x,t,w)}{k^p} + e^{-ikct+i k \psi(w, x)}\sum_{p=0}^N \frac{b_p}{k^p}+ U_{N}(x,t,w)$ 

for successive values of $N=1,2,...$ and equating to 0 in the resulting polynomial of k the sum of terms with a factor $e^{-ikct+ikw \cdot x}$ for equal values of $\frac{1}{k^p}$ leads to the recurrence relations
$$ 2ic(\frac{\partial}{{\partial{t}}} + cw \cdot \frac{\partial}{\partial{x}})a_{j}(x,t,w) = (\frac{\partial^2}{\partial{t^2}}- c^2\Delta)a_{j-1}(x,t,w)  $$

for $j=1,2,...$

Similarly for $b_{0}(x,t,w)$
\begin{equation}
\begin{split}
(\frac{\partial^2}{\partial{t^2}}- c^2\Delta)(e^{-ikct+i k \psi(w, x) } b_{0}(x,t,w)) = \\
-k^2c^2e^{-ikct+i k \psi(w, x) } b_{0}(x,t,w) + e^{-ikct+i k \psi(w, x) }\frac{\partial^2{{b_0}(x,t,w)}}{\partial{t^2}} - \\ 
-2ikce^{-ikct+i k \psi(w, x) }\frac{\partial{{b_0}(x,t,w)}}{\partial{t}} - \\
- c^2( -k^2 (\frac{\partial{\psi(x,w)}}{\partial{x_1}})^2e^{-ikct+i k \psi(w, x) }{b_0}(x,t,w) + \\
+ e^{-ikct+i k \psi(w, x) } \frac{\partial^2{{b_0}(x,t,w)}}{\partial{{x_1}^2}} - \\
-2ike^{-ikct+i k \psi(w, x) } \frac{\partial{\psi(x,w)}}{\partial{x_1}} \frac{\partial{{b_0}(x,t,w)}}{\partial{{x_1}}} - \\
-k^2 (\frac{\partial{\psi(x,w)}}{\partial{x_2}})^2e^{-ikct+i k \psi(w, x) }{b_0}(x,t,w) + e^{-ikct+i k \psi(w, x) }\frac{\partial^2{{b_0}(x,t,w)}}{\partial{{x_2}^2}} -\\ 
-2ike^{-ikct+i k \psi(w, x) } \frac{\partial{\psi(x,w)}}{\partial{x_2}} \frac{\partial{{b_0}(x,t,w)}}{\partial{{x_2}}} ) = 0
\end{split}
\end{equation}
 
Cancelling $e^{-ikct+i k \psi(w, x) }$,  collecting terms and setting to 0 the coefficients for the powers of k in the above polynomial of k leads to the eikonal equation 
$$|\nabla{\psi}|^2=1$$  
and
$$ \frac{\partial{{b_0}(x,t,w)}}{\partial{t}} -c(\frac{\partial{\psi(x,w)}}{\partial{x_1}}\frac{\partial{{b_0}(x,t,w)}}{\partial{{x_1}}} + \frac{\partial{\psi(x,w)}}{\partial{x_2}} \frac{\partial{{b_0}(x,t,w)}}{\partial{{x_2}}})=0$$

and in the general case for $ j \geq 1$ the latter equation has the form
$$ 2ic(\frac{\partial{{b_j}(x,t,w)}}{\partial{t}} -c\nabla{\psi} \cdot \nabla{{b_j}(x,t,w)})  =(\frac{\partial^2}{\partial{t^2}}- c^2\Delta)b_{j-1}(x,t,w) $$

In the recursive construction of $a_N$ and $b_N$, i.e. $p=N$ in the target solution sum, there are no terms to complete the recursive equations for the finite number of terms of order $O(\frac{1}{k^N})$. Therefore,

$$ (\frac{\partial^2}{\partial{t^2}}- c^2\Delta)(U_N(x,t,w)) = O(\frac{1}{k^N})$$
and
$$ U_N|_{\partial{\Omega_1}} = 0 $$

As discussed in \cite{E4}, this implies that $U_N=O(\frac{1}{k^N})$

The support of the constructed solution is in a small neighborhood of a broken ray. 
This neighborhood is defined by the cutoff functions in the solution. We can measure the travel 
time of an approximately linear signal sent from the beginning of the broken ray and received at the endpoint. 
Therefore, we can work directly with broken rays instead of their approximations. 

\section{Numerical Methods for Broken Ray Tomography}\label{algorithms}

The broken ray tomography problem is well-posed\cite{E1}. This key result implies that numerical methods for image 
reconstruction through broken ray tomography should be more accurate and stable compared to the image reconstruction 
methods of classical tomography. Indeed, the classical tomography problem for a domain with an obstacle is ill-posed \cite{Na1} 
and therefore, the well-posedness of the broken ray 
tomography problem is a crucial advantage for numerical methods for broken ray tomography. 
The well-posedness of the broken ray tomography problem follows from \cite{E1} and the estimate from this paper

\begin{equation}\label{BRTEstimate}
\int_{\Omega_0 \backslash \Omega_1} |f(x)|^2 dx \leq C \int_{0}^{l_0} \int_{0}^{2\pi} (  |\frac{\partial{w(x(s),\theta(\phi))}}{\partial{s}}|^2 + |\frac{\partial{w(x(s),\theta(\phi))}}{\partial{\phi}}|^2   ) d{\phi}ds
\end{equation}

The notation for this estimate from \cite{E1,E3} is as follows. Let $\gamma_{x,\theta}$ be a broken or unbroken ray starting on point $T$ in $\partial{\Omega_0}$ with an endpoint $R$ in $\overline{\Omega_0} \backslash \Omega_1$ and direction angle $\theta \in S^1$ at point $x=R$. Let $w(x,\theta)=\int_{\gamma_{x,\theta}} fds$. Consider that the values of $w(x,\theta)$ are known, and equal to $C_{\gamma}$, for all $\theta \in S^1$ when $x \in \partial{\Omega_0}$. In this notation, $\theta(\phi)=(\cos{\phi}, \sin{\phi})$, $0 \leq \phi < 2 \pi$, and $l_0$ is the length of $\partial{\Omega_0}$.

\bigskip

This estimate implies that the solution of the broken ray tomography problem is uniformly bounded in the $L^2$ norm. Indeed, the smooth function $w(x, \theta)$ is defined for every $x$ in the domain $\overline{\Omega_0 \backslash \Omega_1} \subset B_R$, where $B_R = \{x : |x| < R \}$, of the function $f(x)$ and $0 \leq \theta \leq 2 \pi$. Therefore, its derivatives with respect to $x$ and $\theta$ are bounded. Therefore, $\|f\|_{L^2(\Omega_0 \backslash \Omega_1)}$ is bounded.

Consider a perturbation of the input data $w_1(x,\theta)$ by a small amount to $w_2(x,\theta)$. Then the energy
of the error can be estimated by 

\begin{equation}\label{errorestimate1}
\begin{split}
 \int_{\Omega_0 \backslash \Omega_1} |f_1(x)-f_2(x)|^2 dx \leq C \int_{0}^{l_0} \int_{0}^{2\pi} (  |\frac{\partial{w_1(x(s),\theta(\phi))}}{\partial{s}} - \frac{\partial{w_2(x(s),\theta(\phi))}}{\partial{s}}|^2 + \\ 
+ |\frac{\partial{w_1(x(s),\theta(\phi))}}{\partial{\phi}}-\frac{\partial{w_2(x(s),\theta(\phi))}}{\partial{\phi}}|^2   ) d{\phi}ds
\end{split}
\end{equation}

The right hand side of the above inequality is bounded therefore the energy of the error $\int_{\Omega_0 \backslash \Omega_1} |f_1(x)-f_2(x)|^2$ is bounded. 

\bigskip

In developing numerical methods for solving the tomography problem in the presence of a reflecting obstacle we will use this result indirectly.
Instead, we will directly rely on and build the numerical algorithms on the assumptions necessary for the above estimate and required by the theory of 
broken ray tomography. The key assumption and requirement of the theory of tomography in the presence of a reflecting obstacle is that all rays in the 
domain, both broken and unbroken, starting and ending at the observation boundary should be considered in order for the above estimate to hold. 
In other words, the problem is guaranteed to become well-posed when we consider all such rays.

\section{Broken Ray Tomography with the Kaczmarz Method}

In this work we apply the classical Kaczmarz method\cite{K} to the Broken Ray Tomography Problem and show that broken ray tomography can be 
successfully performed with a well-known numerical method. We show through numerical experiments that linear systems corresponding to mixtures 
of broken and unbroken rays have more accurate solutions compared to linear systems corresponding to the same number of unbroken rays only. 
These numerical experiments show that broken ray tomography with the Kaczmarz method in the presence of a reflecting obstacle has much smaller 
error between the original and reconstructed image compared to classical tomography with the Kaczmarz method. 

\bigskip

It is well-known that reflections from an obstacle can improve the accuracy of tomographic imaging\cite{Na2}. 
The unique focus of the numerical solution of the Broken Ray Tomography problem is on determining the composition and properties of finite sets of 
rays that lead to linear systems that have more accurate and stable solutions when solved by the Kaczmarz method. We give a numerical solution of 
the broken ray tomography problem that restricts the  solution of the tomography problem and equivalent problems to the solution of a class of 
linear systems that were obtained from sets of rays with favorable properties. 

\bigskip

I first describe the Kaczmarz method in the context of tomography in the presence of a reflecting obstacle and then I 
study the properties of finite sets of rays for solving the Broken Ray Tomography problem by the Kaczmarz method. 
The Kaczmarz method is the original and key method used in tomographic algebraic reconstruction algorithms and in tomography it 
is referred to as ART. It is very flexible and works with rays with a wide range of known geometries.  

\bigskip

Consider a square domain $M \subset \mathbb{R}^2$ that contains the observable domain $\Omega_0$ i.e. $\Omega_1 \subset \Omega_0 \subset M \subset \mathbb{R}^2$. 
The square M is subdivided into a grid of $N^2$ squares or cells of size d. We define $f(x)=0$ in $\Omega_1$, inside the obstacle, and in 
$M \backslash \Omega_0$, outside the observation boundary, and look for a good approximation of $f(x)$ in each cell of the grid. The value of $f(x)$ is 
considered to be constant in each cell. We arrange linearly the $N^2$ cells into a column vector $f=(f_{1,1},..., f_{N,N})$ where $f_{i,j}$ is the value 
of $f(x)$ in cell $M[i][j]$.

\bigskip

When a ray j intersects cell i of the vector f, then the length of the ray segment that the cell cuts from the ray is the weight of the cell with 
respect to this ray or $w[j][i]$. The matrix of weights for all cells and all rays is denoted as W and has r rows and $N^2$ columns, 
where r is the total number of rays. Let $T_j>0$ be the travel time of ray j. Let $$ \sum_{i=1}^{N^2} w[j][i]f[i] = T_j $$

This is the equation of a hyperplane in $\mathbb{R}^{N^2}$ with normal vector $w_{j}$.
Then the linear system of equations for all rays' travel-times can be expressed as $$Wf=T$$ where T is the column vector of ray travel times. 
This linear system can be very large and even in initial numerical simulations has several thousand equations. 
The Kaczmarz method is an iterative method for solving large linear systems and it starts with an initial guess $f^{(0)}$. 
Then

$$ f^{(i+1)} = f^{(i)} + \frac{(T_{h}-w_{h} \cdot f^{(i)})}{w_{h} \cdot w_{h}}w_{h} $$

where $h = (i \bmod r) + 1$ and r is the number of rays or number of rows of the linear system.

\bigskip

Stefan Kaczmarz proved in \cite{K} the convergence of his iterative method for the solution of large regular linear systems. 
Tanabe proved  in \cite{T} that the Kaczmarz method will converge for any system of linear equations with nonzero rows.
In general, estimating the Kaczmarz method's speed of convergence is still an open problem.

\bigskip

The Kaczmarz method can be applied when the ray geometry is known and it is feasible to compute the intersection of the rays 
with the grid cells. In broken ray tomography the geometry of the broken ray is different from a straight line segment, 
however this geometry is known. In addition, the geometry and location of the obstacle are known and we can compute the 
intersections of the broken rays with the domain's cells outside the obstacle. Given an overdetermined system of rays with 
known geometries that includes both unbroken and broken rays, the Kaczmarz method can therefore be applied in order to reconstruct 
the velocity structure of an environment with a known obstacle. This is the proposed numerical solution of the broken ray tomography problem:

\bigskip
\textbf{
Denote by $\mathbb{A}$ the set of all broken and unbroken rays that start and end at the observation boundary.
Expand the set of rays that are used in the tomographic reconstruction to $\mathbb{A}$ or a discrete approximation
of $\mathbb{A}$.
}
\bigskip

The theory of well-posedness of the broken ray tomography problem is based on consideration of the 
set of all rays, broken and unbroken, that start and end at the observation boundary. In other words, 
we need all rays if we want the reconstruction problem to be well-posed. There is an infinite number of 
such rays and for a numerical solution we need to approximate this condition in order to get as close as 
possible to the requirement for including all rays in the reconstruction. Our goal then is to characterize 
those finite sets of rays that approximate well the set of all rays. We start by defining what it means to 
approximate well. Consider an instance of the broken ray tomography problem with fixed domain $\Omega_0$, 
convex obstacle $\Omega_1$, function f defined in $\Omega_0 \backslash \Omega_1$, and grid and cell size. 
Let s be a set of rays for reconstructing f that start and end at $\partial{\Omega_0}$. When signal travel 
time data is measured along the rays s, this set of rays leads to a linear system l. We solve l by the Kaczmarz 
method. Let $\hat{f}$ be the solution of l. Then $$e(s)=\|f-\hat{f}\|$$ is the error function of s induced by 
this instance of the broken ray tomography problem. In other words, $e(s)$ is the error between the 
reconstructed and original value of f. The Euclidean and max norms or other norms can be chosen. 

\bigskip

Therefore, when solving a given instance of the broken ray tomography problem, we would like to select a set or 
sets s leading to a function e(s) that turns the given instance of the tomography problem into a well-posed one. 
The solution of the broken ray tomography problem found with the discrete set of rays s approximates the solution 
found with $\mathbb{A}$. The theory of broken ray tomography implies that $\mathbb{A}$ leads to a well-posed instance 
of the broken ray tomography problem.

\bigskip

Our goal is to have an efficient method of generating s or selecting from different $s_1, s_2,...$ without solving the 
associated linear systems $l_1, l_2,...$. Let two rays be equivalent if they intersect the same cells of the grid for 
a given instance of the broken ray tomography problem. This equivalence relation on the set of all broken and unbroken 
rays leads to a finite number of equivalence classes of rays. Indeed, each class is equivalent to a discrete ray composed 
from a finite number of cells from the grid. There are clearly a finite number of such discrete rays. This suggests a
strategy for selecting s and a criteria for good approximation of $\mathbb{A}$. Instead of seeking 
to minimize e(s), we select a finite set s such that each equivalence class of rays for the given instance of the 
broken ray tomography problem has at least one of its elements in s. Then s can be considered a finite set that 
approximates $\mathbb{A}$. As a finite set of sets of measure 0, s also has measure 0. For large grids, the number 
of elements in this set is a large but manageable instance of a discrete graph problem:

\bigskip

Consider the cells of the grid that intersect $\partial{\Omega_0}$ as the vertices of a graph $G_1$ and the cells of the grid that
intersect $\partial{\Omega_1}$ as the vertices of a graph $G_2$. A discrete unbroken ray is an edge between two vertices from $G_1$ 
while a discrete broken ray is composed of two edges with one vertice from $G_1$ and a shared vertice from $G_2$.
If the number of vertices of $G_1$ is $V_1$ and the number of vertices of $G_2$ is $V_2$ then the order of the number of elements 
of a discrete approximation of $\mathbb{A}$ is $$O({V_1}^2 + V_1V_2)$$ 

For a given instance of the broken ray tomography problem with an unknown f, 
it is possible to precompute such an approximation of $\mathbb{A}$. 
This is possible because the geometry of all rays for a given instance of the broken ray tomography problem is known because of the 
linearization of ray propagation and the law of reflection. Data is collected for the rays in the set s and the resulting linear system 
solved by the Kaczmarz method. The justification for such a procedure is the improved accuracy and stability due to the well-posedness 
of the broken ray tomography problem. Such a finite set containing representatives from all equivalence classes of rays is still 
very large. In the numerical implementation described in the next chapter rays from this set are selected randomly 
without explicitly generating the whole set.

\section{Numerical Method Implementation}
Numerical methods for broken ray tomography in the presence of an obstacle should in theory be more accurate and stable 
reconstruction methods compared to classical tomographic reconstruction algorithms. This is an elegant result implied 
by the well-posedness of the broken ray tomography problem proved in \cite{E1}. It implies that given a sufficiently 
large overdetermined linear system that corresponds to a large number of unbroken rays, 
we can improve the accuracy of the solution of the system by considering a large overdetermined system corresponding 
to the same number of rays a portion of which are broken. Moreover, broken rays are introduced naturally into the system 
due to reflection.

\bigskip
 
Consider a fixed domain, grid, obstacle, signal wave carrier, and function $f$ to be reconstructed. Consider all sets L 
of broken rays and unbroken rays in the domain. Some sets may contain only unbroken rays, some sets may contain only broken 
rays and other sets may contain both broken and unbroken rays. Each ray element of such a set is determined by its geometry 
and travel time between start and endpoint. Each such set determines a linear system for reconstructing f by tomographic 
reconstruction methods such as the Kaczmarz method. Therefore, to each such set of rays corresponds a linear system, 
or a set of equations, for reconstructing f. We conjecture that improved accuracy and stability of the reconstructed solution 
will become visible for a sufficiently large number of broken and unbroken rays that are uniformly distributed in the set of equivalence 
classes of all rays. This approach approximates the requirements of the theory of broken ray tomography which in turn implies
well-posedness of the tomography problem in the presence of a reflecting obstacle. This section describes the implementation of
our numerical solution of the broken ray tomography problem.

\bigskip
 
We consider a uniform spatial distribution in which the broken and unbroken rays are uniformly distributed in the domain.  
Our goal is to approximate the set of all rays therefore other distributions will favor some elements of the set of all rays 
and exclude other elements of this set.

\bigskip

In order to verify the effectiveness of broken ray tomography with an obstacle, I generated an environment with a known obstacle and 
velocity that varies continuously in the observation region outside the obstacle. The following types of experiments were performed 
in order to compare broken ray tomography with tomography with rays that are straight line segments: a class of experiments for 
a fixed instance of the broken ray tomography problem and different ray sets approximating the set of all rays $\mathbb{A}$,
a class of exeriments where the size of the obstacle was varied, a class of experiments where the ratio of broken and unbroken rays 
was varied.

\bigskip

A good numerical recipe for broken ray tomography should result in a reconstructed $f$ that is closer in some norm to the true $f$. 
The basic algorithm for verifying the effectiveness of broken ray tomography can be summarized as follows:

1. Select a square domain $M$.

2. Partition M into $N^2$ square cells of size $d$. In other words, the square grid M will have $N$ cells of size d per row. 

3. Specify an observation boundary $\partial{\Omega_0}$ as a circle with center in M and configurable radius.

4. Specify a known obstacle in $M$ and represent it with the list of its boundary points.

5. Generate a configurable number of transmitters and receivers located on the observation boundary.

6. Generate broken rays by selecting randomly transmitters and points on the obstacle's boundary that are visible from the transmitters. 
Each pair of transmitter and obstacle boundary point can have exactly one corresponding receiver that receives a signal sent from the 
transmitter that is reflected at the obstacle boundary point according to the law of reflection. 
In this work, we consider specular reflection i.e. reflection in a single direction such that angle of incidence is equal to the angle of 
reflection. We have performed and will report experiments with Lambertian reflection as well.

\bigskip

In this network, there are no other blocking transmitters and receivers on the broken ray segments connecting the reflection point on the 
obstacle with a transmitter and receiver pair. The choice of a circular observation boundary ensures this.

\bigskip

Each transmitter receiver pair is considered only once: this corresponds to an implementation model in which at a given time a
transmitter transmits in only one direction. In order to consider all rays, this restriction should be removed.

7. Generate unbroken rays by selecting random pairs of transmitters and receivers that are not connected by broken rays: this
corresponds to an implementation model in which at a given time a transmitter transmits in only one direction. 
In order to consider all rays, this restriction should be removed. Each pair has exactly one transmitter and exactly one receiver 
and the transmitter and receiver are visible from each other. For example, transmitters and receivers that are connected by a line 
segment that intersects the obstacle are not considered as endpoints of unbroken rays. 

8. Compute the weight matrix W of broken and unbroken ray intersections with the grid M. Elements w[j][i] of the matrix that are
outside the observation boundary and inside the obstacle are set to 0.

9. Choose a continuous function $f(x,y) \in \mathbb{R}^2$. Discretize $f$ by considering its domain as the cells of the grid. 
Set the value of $f$ in each cell of the grid to be constant. The result $f[1]$, $f[2]$,...,$f[N^2]$ are the values of $f$ in 
the grid's cells.

10. Use the known values of $f$ from 9. to find the traveltime of a signal traveling along the broken or unbroken rays from 6. and 7.
The physical model implies that the value of $f$ in a given cell is equal to the reciprocal of the speed of 
ultrasound in the given cell for positive f or 0 when the cell is inside the obstacle or outside the domain $\Omega_0$. 
Therefore, multiplying the length of the section of the ray in the cell to the value of $f$ in the cell gives the traveltime of ultrasound in the cell. By knowing the geometry of a ray, we know which cells of the grid the ray will 
intersect and we compute and add the traveltimes for each intersected cell to find the sum equal to the total traveltime of an 
ultrasonic signal along a given ray. This numerical integration along the path of ray j gives the traveltime $T_j$ of the ray for the 
chosen $f$. When this procedure is done for all rays, we know the vector $T$ of travel times for all rays. 

11. Steps 8. and 10. give the coefficients matrix and right hand side of a linear system $Wx=T$. This system is very large and is solved with the Kaczmarz or other methods to find the value of the vector $x$. The vector $x$ has the same meaning as the vector $f$, however, $x$ is reconstructed numerically while $f$ is chosen a priori. 

12. Now we can compare $f$ and $x$ and see how well the reconstructed value matches the true value of f.

\bigskip

The numerical method is implemented in an original and custom Java software tomography framework called Euler. 
The program from \cite{G} is another Java tomography framework for algebraic reconstruction that I have studied.

\section{Experimental Results}

Table \ref{BRTARTComparison} summarizes experimental results for tomographic reconstruction of a 
function $f(x,y)=K\sqrt{(x-x_0)^2+(y-y_0)^2}$ in a square domain with 4096 cells, or 64 cells per row 
where $(x_0, y_0)$ is the center of the domain. The size of each cell is 13 points or units. 
The domain contains a square obstacle with 30 cells per row, and the circular observation boundary 
enclosing the obstacle has radius 350 units. There are 512 transmitters and 512 receivers along the 
boundary at equal angles between neighbor transmitters and the center of the observation boundary and 
between neighbor receivers and the center of the observation boundary. Each row of the table corresponds 
to a numerical solution for the function $f(x,y)$ obtained with the same number of rays. 
Row 1 of  Table \ref{BRTARTComparison} shows results for an experiment with 126050 unbroken rays. These rays
are an approximation of all possible unbroken rays between the transmitters and receivers that do not intersect 
the obstacle. Row 2 of Table \ref{BRTARTComparison} shows tomographic reconstruction for the same domain, obstacle and 
observation boundary with broken and unbroken rays. The total number of broken and unbroken rays in the 
domain is much larger compared to the total number of unbroken rays. In order to compare the effectiveness 
of broken ray tomography for the same number of rays, we take a random sample that has the same size 126050 
as the number of unbroken rays in the solution with unbroken rays only. The ratio of broken and unbroken 
rays is 1:1 i.e. there are 63025 broken and 63025 unbroken rays in the sample that are chosen approximately uniformly.

\bigskip

The error column shows the average error per cell between the original and reconstructed value of f. 
The number of iterations column shows the number of iterations of the Kaczmarz method before it converges to a solution. 
The criteria for  convergence is that for at least two consecutive steps the currently computed value of f must be within a 
radius of convergence from the value of f computed at the previous step. I use the stronger max norm in addition to the classical 
Euclidean distance. In conclusion, the experimental results indicate that in the presence of an obstacle broken ray tomography 
is a more accurate approach to tomographic reconstruction compared to classical tomography. For example, visual inspection of the reconstructed images in 
figure \ref{126050unbrokenrays_reconstructed} and figure \ref{126050brokenandunbrokenrays_reconstructed} shows that compared to 
the image reconstructed with unbroken rays only, the reconstructed broken ray tomography image is significantly more accurate.

\bigskip

Table \ref{BRTData1} shows the performance of our numerical solution of the broken ray tomography problem and compares it to the 
performance of ART for a fixed instance of the tomography problem and ten different ray sets for each method. Reconstruction error
increases when the fraction of unbroken rays in a ray set is close to 1 as shown by the results in Table \ref{BRTData2} for the 
same instance of the broken ray tomography problem as in Table \ref{BRTData2}.
Table \ref{BRTData3} compares the performance of the two methods when the size of the obstacle is varied.

\bigskip

In conclusion, we have observed in numerical experiments that broken ray tomography is on average three times more accurate compared to 
tomography without reflection. The theory of broken ray tomography predicts that the accuracy and advantages of broken ray tomography 
are much larger.

\clearpage

\begin{table}[h]
\begin{tabular}{|c|ccc|}
	\hline
Tomography Type &  Number of Rays &   Error  & Iterations  \\
  \hline
ART & 126050   & 1.80484955E-4 & 37144  \\
BRT with Kaczmarz method & 126050 & 4.820056689E-5 & 22728 \\
  \hline
\end{tabular}
\caption{Comparison between ART and Broken Ray Tomography with the Kaczmarz method. 
We compare tomography without reflection and tomography with reflection in the presence of a reflecting obstacle.\label{BRTARTComparison}}
\end{table}

\bigskip

\begin{table}[h]
\begin{tabular}{|c|cc|cc|}
	\hline
Experiment & ART Error  & ART Iterations & BRT Error  & BRT Iterations  \\
  \hline

1 & 1.104141e-004 & 71502 & 1.980839e-005 & 77928 \\ 
2 & 1.153358e-004 & 69675 &    9.159289e-005 & 52365 \\
3 & 1.582274e-004 & 46328 &  2.191676e-005 & 75798 \\
4 & 9.906414e-005 & 93348 &   1.629515e-005 & 89842 \\
5 & 1.579209e-004 & 45012 &    2.107154e-005 & 77348 \\
6 & 1.511201e-004 & 39476 &   2.113383e-005 & 80295 \\
7 & 1.349709e-004 & 54259 &   2.139517e-005 & 78698 \\
8 & 1.416369e-004 & 57246 &   1.001323e-004 & 40882 \\
9 & 1.313073e-004 & 63075 &   1.732388e-005 & 83562 \\
10 & 1.383719e-004 & 52971 &   2.189941e-005 & 75628 \\

  \hline

Average & 1.338370e-004  & 59289.2 & 3.525693e-005 & 73234.6 \\

  \hline
\end{tabular}
\caption{Error and number of iterations for BRT for a fixed number of 126050 rays with 50\% broken and 50\% unbroken rays.
The average error is 3.525693e-005 and the average number of iterations of the Kaczmarz method for finding a solution is
73234.600000. The results for ART tomography with 126050 unbroken rays are shown in the left two columns of the table. 
The average error for ART is 1.338370e-004 and the average number of iterations is 59289.200000.
\label{BRTData1}}
\end{table}

\clearpage

\bigskip

\begin{table}[h]
\begin{tabular}{|c|cc|}
	\hline
Fraction of Unbroken Rays & BRT Error & BRT Iterations  \\
  \hline

0.500000 & 9.209899e-005 & 58851 \\ 
0.550000 & 1.692470e-005 & 84958 \\
0.600000 & 2.651309e-005 & 60223 \\
0.650000 & 5.372332e-005 & 46297 \\
0.700000 & 2.299121e-005 & 57166 \\
0.750000 & 3.372089e-005 & 41231 \\
0.800000 & 2.778578e-005 & 44129 \\
0.850000 & 3.443016e-005 & 33874 \\
0.900000 & 4.696702e-005 & 43283 \\
0.950000 & 1.286371e-004 & 47213 \\

  \hline

Average & 4.837922e-005  & 51722.5 \\

  \hline
\end{tabular}
\caption{Performance of broken ray tomography for a fixed instance of the tomography problem with ray sets of 126050 rays with 
different fractions of broken and unbroken rays. When the fraction of unbroken rays is close to 1 the reconstruction error 
increases. The average error is 4.837922e-005 and the average number of iterations 51722.5. 
\label{BRTData2}}
\end{table}

\bigskip

\begin{table}[h]
\begin{tabular}{|c|cc|cc|}
	\hline
Side Length & ART Error  & ART Iterations & BRT Error  & BRT Iterations  \\
  \hline

130 & 1.722081e-004 & 57490 & 4.913426e-005 & 83495 \\
156 & 2.532327e-004 & 41801 & 4.810847e-004 & 36541 \\
182 & 2.786448e-004 & 37833 & 3.308767e-005 & 85148 \\
208 & 2.374432e-004 & 55282 & 2.977520e-005 & 84985 \\
234 & 1.696454e-004 & 135341 & 3.608453e-005 & 79498 \\
260 & 2.218613e-004 & 76752  & 2.544549e-005 & 84010 \\
286 & 1.888231e-004 & 91506  & 2.249251e-005 & 84914  \\
312 & 1.821901e-004 & 79507  & 2.492110e-005 & 81309 \\
338 & 2.053815e-004 & 51063 & 1.416251e-004 & 31923 \\
364 & 1.743495e-004 & 48376 & 2.129221e-005 & 81740 \\

  \hline

Average & 2.083780e-004 & 67495.1 & 8.649428e-005 & 73356.3 \\

  \hline
\end{tabular}
\caption{Error and number of iterations for BRT for a fixed number of 126050 rays with 50\% broken and 50\% unbroken rays.
Tomographic reconstruction is performed in ten experiments with different side lengths of the square obstacle.
The average error is 8.649428e-005 and the average number of iterations of the Kaczmarz method for finding a solution is
73356.3. The results for ART tomography with 126050 unbroken rays are shown in the left two columns of the table. 
The average error for ART is 2.083780e-004 and the average number of iterations is 67495.1.
\label{BRTData3}}
\end{table}

\clearpage

\begin{figure}
\begin{center}
\includegraphics[scale=0.50]{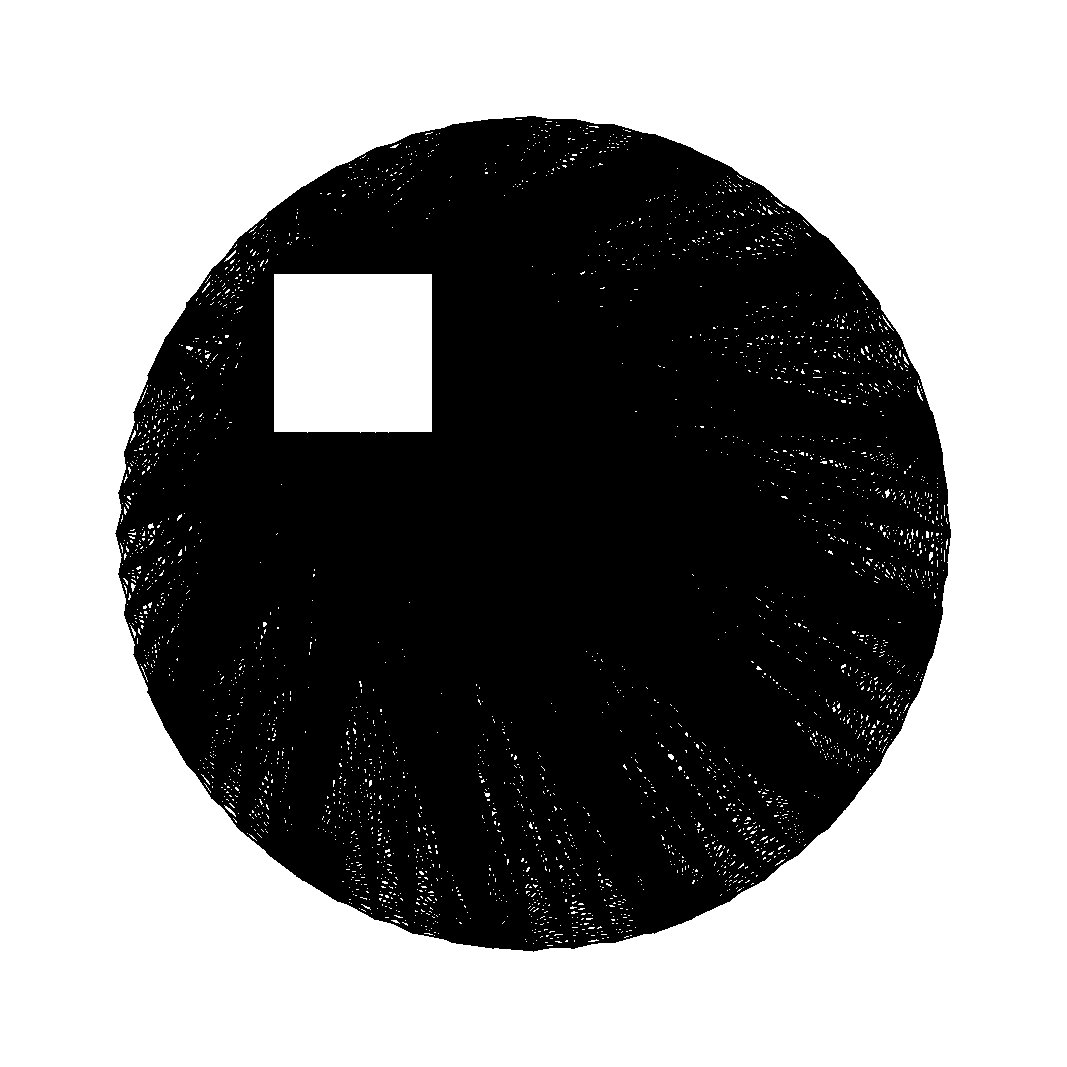} 
\caption{Tomography with broken rays and unbroken rays}
\label{allrays}
\end{center}
\end{figure}


\begin{figure}
\begin{center}
\includegraphics[scale=0.50]{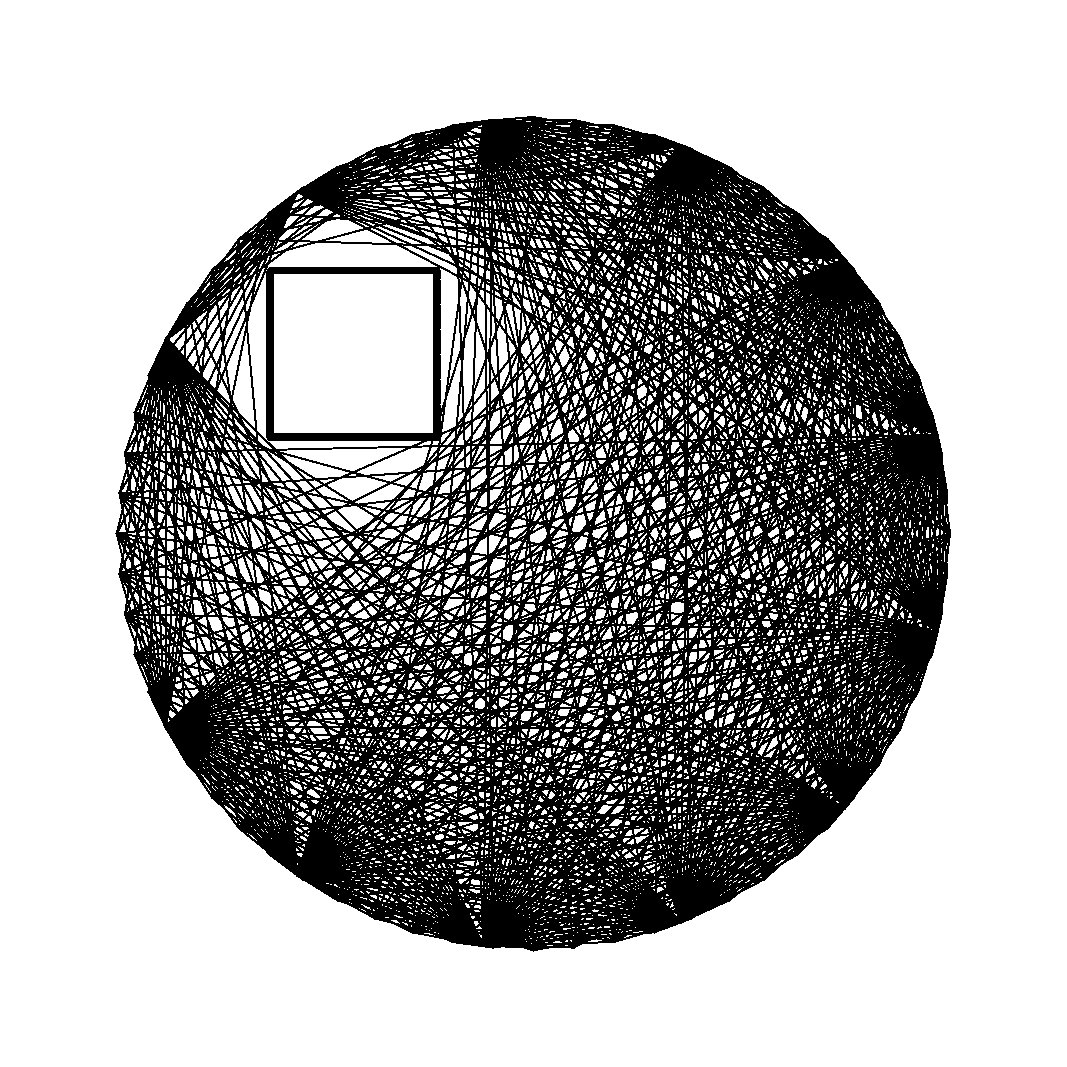} 
\caption{Tomography with unbroken rays only}
\label{unbrokenrays}
\end{center}
\end{figure}

\begin{figure}
\begin{center}
\includegraphics[scale=0.50]{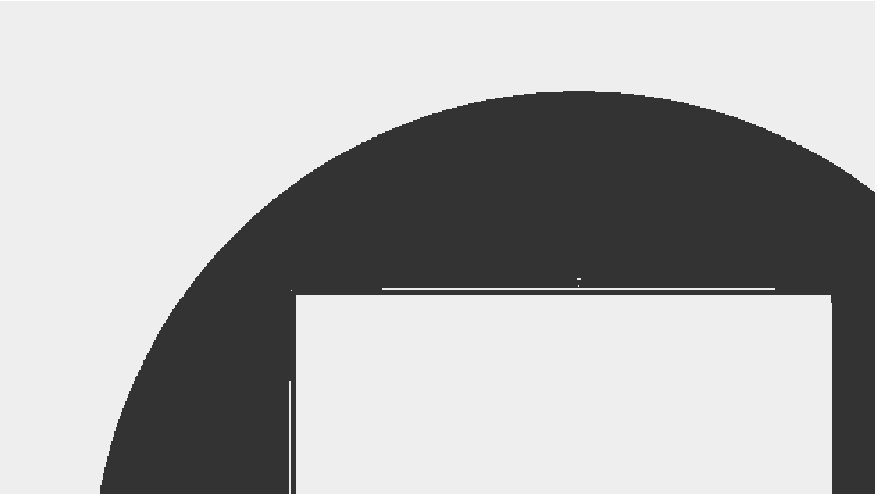} 
\caption{Tomography with 126050 unbroken rays only}
\label{126050unbrokenrays}
\end{center}
\end{figure}

\begin{figure}
\begin{center}
\includegraphics[scale=0.50]{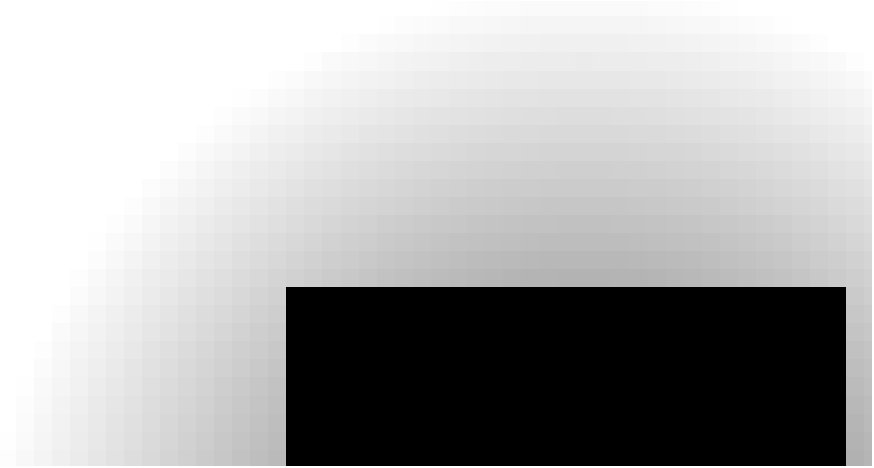} 
\caption{Test case for tomography with 126050 unbroken rays only. Original value of f is continuously varying grey around the black square obstacle.}
\label{126050unbrokenrays_original}
\end{center}
\end{figure}

\begin{figure}
\begin{center}
\includegraphics[scale=0.50]{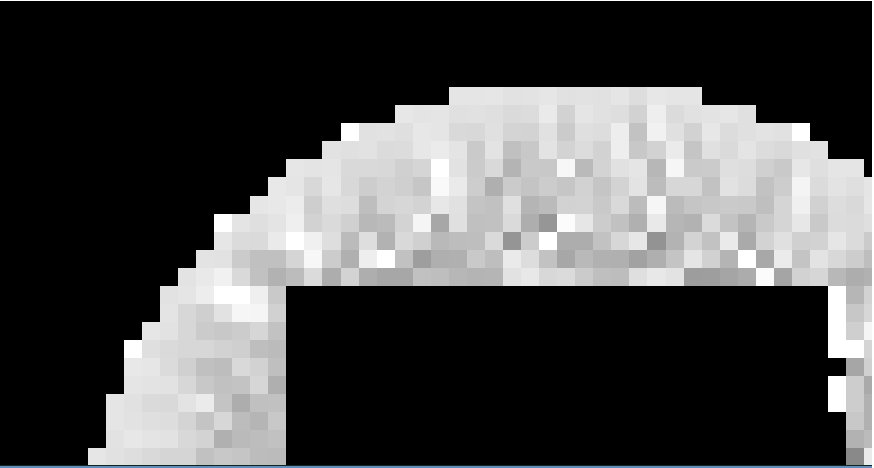} 
\caption{Tomography with 126050 unbroken rays only. Reconstructed f is inside the circular observation boundary and outside the black square obstacle.}
\label{126050unbrokenrays_reconstructed}
\end{center}
\end{figure}

\begin{figure}
\begin{center}
\includegraphics[scale=0.50]{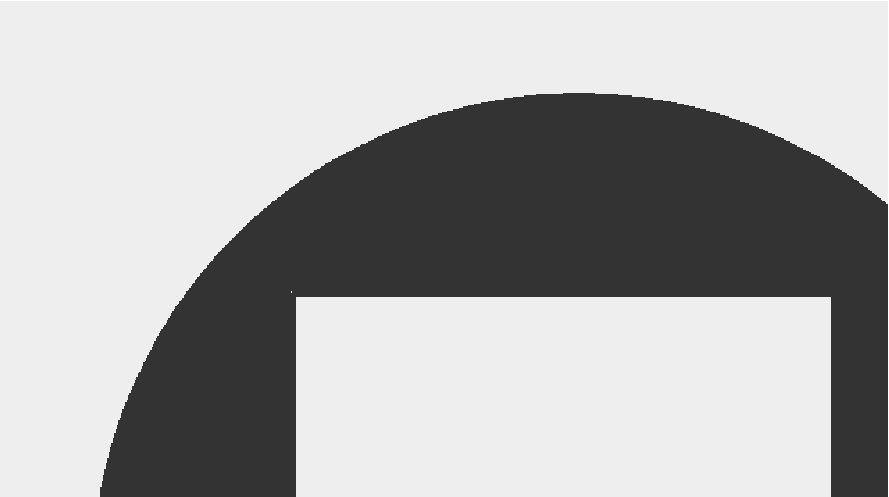} 
\caption{Tomography with 63025 broken and 63025 unbroken rays.}
\label{126050brokenandunbrokenrays}
\end{center}
\end{figure}

\begin{figure}
\begin{center}
\includegraphics[scale=0.50]{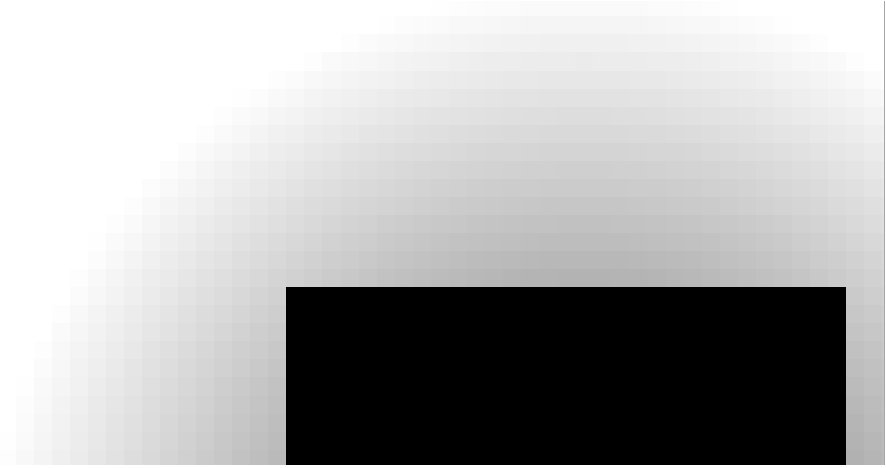} 
\caption{Test case for tomography with 63025 broken and 63025 unbroken rays. Original value of f is continuously varying grey around the black square obstacle.}
\label{126050brokenandunbrokenrays_original}
\end{center}
\end{figure}

\begin{figure}
\begin{center}
\includegraphics[scale=0.50]{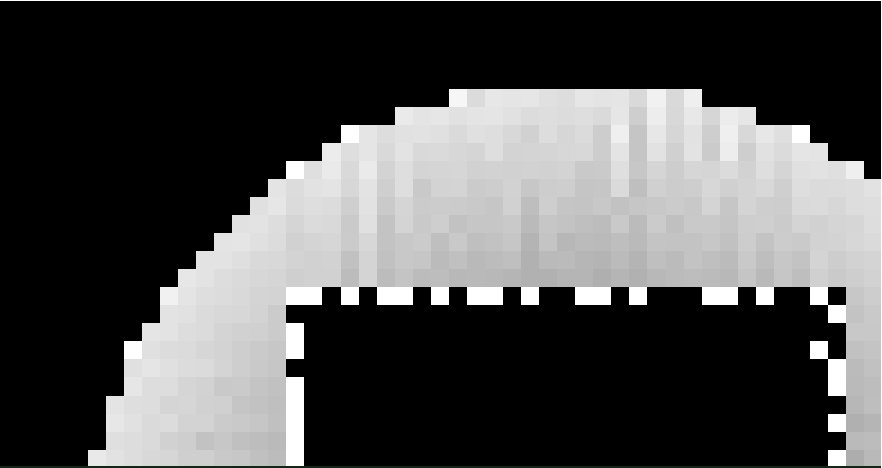} 
\caption{Tomography with 63025 broken and 63025 unbroken rays. Reconstructed f is inside the circular observation boundary and outside the black square obstacle.}
\label{126050brokenandunbrokenrays_reconstructed}
\end{center}
\end{figure}

\clearpage
\newpage

\section{Tomography in Three-Dimensional Domains with Obstacles}

The classical tomography problem for a three-dimensional domain without obstacles can be solved by taking two-dimensional plane cuts and solving the two-dimensional tomography 
problem in each of the planar cuts. This approach is not generally applicable for broken ray tomography because for a given plane cut a ray that is in the
plane of the cut before the reflection point is not guaranteed to be reflected in the same plane. Professor Eskin suggested a method for solving the 
three-dimensional tomography problem in a domain with one convex obstacle by considering plane cuts of $\Omega_0$ that do not intersect the obstacle. 
The tomography problem in each of these planes is well-posed and can be solved by the Radon transform and the algorithms of classical tomography 
\cite{Na1,K,F,NW}.  It is still interesting and practical to consider a solution of the three-dimensional tomography problem in a 
domain with an obstacle by plane cuts that intersect the obstacle. This approach is easier for applications because it is easier to 
choose the planes that cover the domain if we allow planes that intersect the obstacle.

\bigskip

We can consider solving the three-dimensional broken ray tomography problem when the obstacle is a cylinder or a parallelipiped by taking parallel plane 
cuts that are perpendicular to the cylinder's axis or to one of the sides of the parallelipiped. 
All reflected rays will remain in the same plane as their respective incident rays. In the resulting plane cuts, 
the two dimensional obstacle will be a circle or a rectangle. Rectangles and circles have piece-wise smooth boundaries 
required for obstacles by the theorems of broken ray tomography. It is an interesting question then whether generalized 
cylinders are the only three-dimensional shapes that obstacles can have so that the three-dimensional broken ray tomography 
problem can be solved by two-dimensional plane cuts which reduce the problem to the two-dimensional 
broken ray tomography problem in the union of the cuts that contains the whole obstacle.

\begin{proposition}\label{brt_3d}
Define $\partial{\Omega_1} \in \mathbb{R}^3$ by $F(x,y)=0$ where $F(x,y)$ is a smooth function that is independent of z and such that 
$\Omega_1$ is convex and has a smooth boundary. The Broken Ray Tomography Problem in the compact domain $\Omega_0 \subset \mathbb{R}^3$ 
and for an obstacle $\Omega_1 \subset \Omega_0$ can be solved by plane cuts of $\Omega_0$ a subset of which contains the obstacle 
$\Omega_1$. The cuts in that subset reduce the problem to the two dimensional broken ray tomography problem.
\end{proposition}

\begin{proof}
Let $\Omega_0 \subset \bigcup \Pi_{z}$  where $\Pi_z$ are planes parallel to the xy plane. Reconstruct the restriction of $f$ to 
$\Pi_{z}$ and therefore $f(x,y,z)$ in $\Pi_{z}$ by two-dimensional broken ray tomography in $\Pi_{z}$; when $\Pi_z$ does not intersect $\Omega_1$
f is determined by the Radon transform and can be reconstructed via the algorithms of classical tomography. Reconstruction by reduction to the two-dimensional broken ray tomography problem can be done because 
the surface normal at any point $(x_1,y_1,z)$ on the boundary $\partial{\Omega_1}$ is colinear with 
$$\nabla F(x_1,y_1)=(\frac{\partial{F}}{\partial{x}}(x_1,y_1), \frac{\partial{F}}{\partial{y}}(x_1,y_1), 0)$$ and does not have a z component 
and is always contained in $\Pi_z$. Therefore, all incident rays in $\Omega_0 \bigcap \Pi_{z}$ will be reflected in $\Pi_z$ because $\Pi_z$ contains 
the incident ray and surface normal $\nabla F(x_1,y_1)$ at the reflection point for that ray. Unbroken rays in $\Pi_{z}$ will remain in $\Pi_z$ by 
definition therefore $\Omega_0 \bigcap \Pi_{z}$ contains all broken and unbroken rays that start on $\partial{\Omega_0} \bigcap \Pi_z$ and that have 
their endpoints in $\partial{\Omega_0}$. $\Pi_z \bigcap \partial{\Omega_1}$ is smooth and $\Pi_z \bigcap \Omega_1$ convex because $\partial{\Omega_1}$ 
is smooth and $\Omega_1$ convex and similarly $\Pi_z \bigcap \Omega_0$ is compact. Therefore $f$ can be reconstructed by solving the well-posed two 
dimensional broken ray tomography problem in $\Pi_z \bigcap \Omega_0$.

The reconstructed restriction of $f$ is unique in the plane $\Pi_{z}$ and the function $f(x,y,z)$ is well-defined in $\Omega_0$ 
because the sets $\Pi_z$ are disjoint and each point $(x,y,z)$ is in exactly one such set. 
\end{proof}

There are other shapes that allow reconstruction with plane cuts through the obstacle. For example, obstacle surfaces that are defined implicitly by 
functions that are independent of x or y can be cut by planes that are orthogonal to the x and y axis respectively. It is an interesting problem what all
the shapes are for an obstacle so that the three dimensional broken ray tomography problem can be solved by reducing it to the two dimensional broken
ray tomography problem in planes intersecting the obstacle and the union of which contains the obstacle. 

One application of Proposition \ref{brt_3d} is three dimensional tomographic reconstruction of the speed of sound in a domain that contains a 
cylindrical pipe. Imagine slices that cut the domain and are perpendicular to the pipe. The thickness of each of these slices is small and they can 
be considered two dimensional. Moreover, the speed of sound in each of the slices can be considered close to a constant and this allows reconstruction by
the numerical method from the previous chapter. The reconstructed values are combined and give f in the three dimensional region around the pipe. 


\section*{Acknowledgments}

I would like to thank my wife for her support, comments and suggestions.
I would like to thank Professor Gregory Eskin for suggesting this problem and for his continuous guidance. 


\bibliographystyle{dcu}
\bibliography{BRT}


\end{document}